\documentclass[11pt]{conm-p-l}
\usepackage{float, amssymb, amsmath, anysize, colortbl, tikz}

\marginsize{2cm}{2cm}{1cm}{2.5cm}

\usetikzlibrary{shapes}

\theoremstyle{plain}
\newtheorem{theorem}{Theorem}[section]
\newtheorem{lemma}[theorem]{Lemma}

\theoremstyle{remark}
\newtheorem{defn}[theorem]{Definition}

\newcommand\GF{\mathbb{GF}}
\newcommand\PG{\mathsf{PG}}
\newcommand\Q{\mathsf{Q}}
\renewcommand\le{\leqslant}

\newcommand{\N}{\mathsf{N}}
\newcommand{\T}{\mathsf{T}}

\title{Low dimensional models of the finite split Cayley hexagon}

\author{John Bamberg}
\address{ %
Centre for the Mathematics of Symmetry and Computation\\
School of Mathematics and Statistics\\
The University of Western Australia\\
35 Stirling Highway, Crawley, W.A. 6009, Australia.}
\email{John.Bamberg@uwa.edu.au}

\author{Nicola Durante}
\address{ %
Dipartimento di Matematica ed Applicazioni\\
Universit\`a di Napoli ``Federico II''\\
Via Cintia, 80126 Naples, Italy.
}
\email{ndurante@unina.it}


\begin{document}

\begin{abstract}
We provide a model of the split Cayley hexagon arising from the Hermitian surface
$\mathsf{H}(3,q^2)$, thereby yielding a geometric construction of the Dickson group
$\mathsf{G}_2(q)$ starting with the unitary group $\mathsf{SU}_3(q)$.
\end{abstract}

\maketitle

\section{Introduction}\label{sect:intro}

A generalised polygon $\Gamma$ is a point-line incidence structure such that the incidence graph is
connected and bipartite with girth twice that of its diameter. If the valency of every vertex is at
least $3$, then we say that $\Gamma$ is \textit{thick}, and it turns out that the incidence graph is
then biregular\footnote{That is, there are two constants $k_1$ and $k_2$ such that the valency of
  each vertex in one bipartition is $k_1$, and the valency of each vertex in the other bipartition
  is $k_2$.}.  By a famous result of Feit and Higman \cite{FeitHigman}, a finite \textit{thick
  generalised polygon} is a complete bipartite graph, projective plane, generalised quadrangle,
generalised hexagon or generalised octagon.  There are many known classes of finite projective
planes and finite generalised quadrangles but presently there are only two known families, up to
isomorphism and duality, of finite generalised hexagons; the \textit{split Cayley hexagons} and the
\textit{twisted triality hexagons}.

The split Cayley hexagons $\mathcal{H}(q)$ are the natural geometries for Dickson's group
$\mathsf{G}_2(q)$, and they were introduced by Tits \cite{Tits59} as the set of points of the
parabolic quadric $\Q(6,q)$ and an orbit of lines of $\Q(6,q)$ under $\mathsf{G}_2(q)$.  If $q$ is
even, then the polar spaces $\mathsf{W}(5,q)$ and $\Q(6,q)$ are isomorphic geometries, and hence
$\mathcal{H}(q)$ can be embedded into a five-dimensional projective space.
Thas and Van Maldeghem \cite{ThasVM98} proved that if $\mathcal{H}$ is a finite thick generalised hexagon
embedded\footnote{We will not discuss the various meanings of ``embedding'' here, but
instead refer the interested reader to \cite{ThasVM96,ThasVM98}.} into the projective space $\PG(d,q)$, then $d\le 7$
and this embedding is equivalent to one of the standard models of the known generalised hexagons.
So in particular, it is impossible to embed the split Cayley hexagon $\mathcal{H}(q)$
into a three-dimensional projective space.
However, there is an elegant model of $\mathcal{H}(q)$ which begins with geometric structures lying
in $\PG(3,q)$, and it is equivalent to the model provided by Cameron and Kantor
\cite[Appendix]{CameronKantor79}:

\begin{theorem}[Cameron and Kantor (paraphrased) \cite{CameronKantor79}]\label{thm:constL3}\ \\
Let $(p,\sigma)$ be a point-plane anti-flag of $\PG(3,q)$ and let $\Omega$ be a set of $q
(q^2-1)(q^2+q+1)$ parabolic congruences\footnote{A \textit{pencil} of lines of $\PG(3,q)$ refers to
  the set of lines passing through a point, lying on a plane.  A set of $q^2+q+1$ lines concurrent
  with a common line $\ell$, no two of which meet in a point not on $\ell$, is called a
  \textit{parabolic congruence}, and the line $\ell$ is its \textit{axis}. The image of a parabolic congruence under the Klein
  correspondence yields a $3$-dimensional quadratic cone of $\Q^+(5,q)$, and vice-versa (see
  \cite[p. 30]{Hirschfeld85}), and so a parabolic congruence is a union of $q+1$ pencils sharing a
  line.}  each having axis not incident with $p$ or $\sigma$, but having a pencil of lines with one
line incident with $p$ and another incident with $\sigma$.  Suppose that for each pencil
$\mathcal{L}$ with vertex not in $\sigma$ and plane not incident with $p$, there are precisely $q+1$ elements
of $\Omega$ containing $\mathcal{L}$, whose union are the lines of some linear complex (i.e., the
lines of a symplectic geometry $\mathsf{W}(3,q)$).  Then the following incidence structure $\Gamma$
is isomorphic to the split Cayley hexagon $\mathcal{H}(q)$.
\begin{table}[H]
\begin{tabular}{ll}\hline
  \textsc{Points:} & (a) Lines of $\mathsf{PG}(3,q)$.\\
  &  (b) Pencils with a vertex not in $\sigma$ and plane not incident with $p$.\\
  \textsc{Lines:} & (i) Pencils with a vertex in $\sigma$ and plane through $p$.\\
  &  (ii) Elements of $\Omega$.\\
  \hline \\
\end{tabular}
\begin{minipage}{16cm}
An element $\ell$ of type (a) is incident with an element $\mathcal{P}$ of type (i) if $\ell$ is an
element of $\mathcal{P}$.  If $\mathcal{C}$ is an element of type (ii), then $\ell$ is incident with
$\mathcal{C}$ if $\ell$ is the axis of $\mathcal{C}$.  Elements of type (i) and (b) are never
incident. The containment relation defines incidence between elements of type (b) and (ii).
\end{minipage}

\end{table}
\end{theorem}

The central result of this note is a unitary analogue of this model.

\begin{theorem}\label{thm:constU3}\ \\
Let $\mathcal{O}$ be a Hermitian curve of $\mathsf{H}(3,q^2)$ and let $\Omega$ be a set of Baer
subgenerators with a point in $\mathcal{O}$, such that every point of
$\mathsf{H}(3,q^2)\backslash\mathcal{O}$ is on $q+1$ elements of $\Omega$ spanning a Baer subplane.
Then the following incidence structure $\Gamma$ is a generalised hexagon of order
$(q,q)$. \begin{table}[H]
\begin{tabular}{ll}\hline
  \textsc{Points:} & (a) Lines of $\mathsf{H}(3,q^2)$.\\
  &  (b) Affine points of $\mathsf{H}(3,q^2)\backslash\mathcal{O}$.\\
  \textsc{Lines:} & (i) Points of $\mathcal{O}$.\\
  &  (ii) Elements of $\Omega$.\\
\textsc{Incidence:} & Inclusion or inherited incidence.\\
\hline
\end{tabular}
\end{table}
\noindent Moreover, $\Gamma$ is isomorphic to the split Cayley hexagon $\mathcal{H}(q)$.
\end{theorem}

The proof that $\Gamma$ is a generalised hexagon is presented in Section \ref{proofconstU3}.
Note that the lines of type (i) form a spread of $\mathcal{H}(q)$.
There exists a natural candidate for $\Omega$ which we explain in detail in Section \ref{sect:partition},
and it is essentially the only one (Theorem \ref{thm:omega}), and this implies the ultimate result
that $\Gamma$ is isomorphic to $\mathcal{H}(q)$.

By the deep results of Thas and Van Maldeghem \cite{ThasVM96,ThasVM98} and Cameron and Kantor
\cite{CameronKantor79}, if a set of points $\mathcal{P}$ and lines $\mathcal{L}$ of
$\mathsf{PG}(6,q)$ form a generalised hexagon, then it is isomorphic to the split Cayley
hexagon $\mathcal{H}(q)$ if $\mathcal{P}$ spans $\mathsf{PG}(6,q)$ and for any point
$x\in\mathcal{P}$, the points collinear to $x$ span a plane. A similar result was proved recently by
Thas and Van Maldeghem \cite{ThasVM08}, by foregoing the assumption that $\mathcal{P}$ and
$\mathcal{L}$ form a generalised hexagon, and instead instituting the following five axioms: (i) the
size of $\mathcal{L}$ is $(q^6-1)/(q-1)$, (ii) every point of $\PG(6,q)$ is incident with either $0$
or $q+1$ elements of $\mathcal{L}$, (iii) every plane of $\PG(6,q)$ is incident with $0$, $1$ or
$q+1$ elements of $\mathcal{L}$, (iv) every solid of $\PG(6,q)$ contains $0$, $1$, $q+1$ or $2q+1$
elements of $\mathcal{L}$, and (v) every hyperplane of $\PG(6, q)$ contains at most $q^3+3q^2+3q$
elements of $\mathcal{L}$.

One could instead characterise the split Cayley hexagon viewed as points and lines of the parabolic
quadric $\Q(6,q)$, and the best result we have to date follows from a result of Cuypers and
Steinbach \cite[Theorem 1.1]{CuypersSteinbach}:

\begin{theorem}[Cuypers and Steinbach \cite{CuypersSteinbach} (paraphrased)]\label{thm:Q6embedding}
Let $\mathcal{L}$ be a set of lines of $\Q(6,q)$ such that every point of $\Q(6,q)$ is incident with
$q+1$ lines of $\mathcal{L}$ spanning a plane, and such that the concurrency graph of $\mathcal{L}$
is connected.  Then the points of $\Q(6,q)$ together with $\mathcal{L}$ define a generalised hexagon isomorphic to
the split Cayley hexagon $\mathcal{H}(q)$.
\end{theorem}

In Section \ref{sect:hexagon} we will give an elementary proof of Theorem \ref{thm:Q6embedding} by using Theorem \ref{thm:omega}.

\subsection*{Some remarks on notation:}
In this paper, the \textit{relative norm} and \textit{relative trace} maps will be defined for the quadratic extension
$\GF(q^2)$ over $\GF(q)$.
The relative norm $\N$ is the
multiplicative function which maps an element $x\in\GF(q^2)$ to the
product of its conjugates of $\GF(q^2)$ over $\GF(q)$. That is, $\N(x)=x^{q+1}$.
The relative trace is instead the sum of the conjugates, $\T(x):=x+x^q$.

%
%

\section{The 3-dimensional Hermitian surface and its Baer substructures}\label{sect:hermitian}

The two (\textit{classical}) generalised quadrangles of particular importance in this note are
$\mathsf{H}(3,q^2)$ and $\Q^-(5,q)$. First there is the incidence structure of all points and lines
of a non-singular Hermitian variety in $\PG(3,q^2)$, which forms the generalised quadrangle
$\mathsf{H}(3,q^2)$ of order $(q^2,q)$. Its point-line dual is isomorphic to the geometry of points
and lines of the elliptic quadric $\Q^-(5,q)$ in $\PG(5,q)$, which yields a generalised quadrangle
of order $(q,q^2)$ (see \cite[3.2.3]{FGQ}). To construct $\mathsf{H}(3,q^2)$ given a prime power
$q$, we take a non-degenerate Hermitian form such as
$$\langle X,Y\rangle = X_0Y_0^{q}+X_1Y_1^{q}+X_2Y_2^{q}+X_3Y_3^{q}$$ and the totally isotropic
subspaces of the ambient projective space, with respect to this form.  Most of the material
contained in this section can be found in Barwick and Ebert's book \cite{BarwickEbert} and
Hirschfeld's book \cite[Chapter 7]{Hirschfeld98}.

Every line of $\mathsf{PG}(3,q^2)$ is (i) a \textit{generator} (i.e., totally isotropic line) of
$\mathsf{H}(3,q^2)$, (ii) meets $\mathsf{H}(3,q^2)$ in one point (i.e., a tangent line), or (iii) meets $\mathsf{H}(3,q^2)$
in a Baer subline (also called a \textit{hyperbolic} line). A \textit{Baer subline} of the projective line $\PG(1, q^2 )$ is a subset of
$q+1$ points in $\PG(1, q^2 )$ which form a $\GF(q)$-linear subspace.  We may also speak of \textit{Baer subplanes} and \textit{Baer
  subgeometries} of $\PG(3,q^2)$ as sets of points giving rise to projective subgeometries
isomorphic to $\PG(2,q)$ and $\PG(3,q)$ respectively. A \textit{Baer subgenerator} of
$\mathsf{H}(3,q^2)$ is a Baer subline of a generator of $\mathsf{H}(3,q^2)$. We will often use the
fact that three collinear points determine a unique Baer subline (\cite[Theorem 2.6]{BarwickEbert}) and
a planar quadrangle determines a unique Baer subplane (\cite[Theorem 2.8]{BarwickEbert}). In particular,
if $b$ and $b'$ are two Baer sublines of $\PG(2,q^2)$ sharing a point, but not spanning the same line,
then there is a unique Baer subplane containing both $b$ and $b'$. We say that it is the \textit{Baer subplane
spanned by $b$ and $b'$}.

One class of important objects for us in this paper will be the degenerate Hermitian curves of rank
2. Suppose we have a fixed hyperplane, $\pi: X_3=0$ say, meeting $\mathsf{H}(3,q^2)$ in a Hermitian
curve $\mathcal{O}$.  Let $\ell$ be a generator of $\mathsf{H}(3,q^2)$. Then the polar planes of the
points on $\ell$ meet $\pi$ in the $q^2+1$ lines through $L:=\ell\cap\mathcal{O}$. Now suppose we have a
Baer subgenerator $b$ contained in $\ell$, and containing the point $L$.  Then the polar planes of
the points of $b$ meet $\pi$ in $q+1$ lines through the point $L$ giving a \textit{dual Baer subline} of
$\pi$ with \textit{vertex} $L$. Moreover, the points lying on this dual Baer subline define a
variety with Gram matrix $U$; a Hermitian matrix of rank $2$. So they correspond to solutions of
$XU(X^q)^T=0$ where $U$ satisfies $U^q=U^T$. For example, if we consider a point $P$ in $\pi$,
say $(1,\omega,0,0)$ where $\N(\omega)=-1$, and two points $A:(a_0,a_0\omega,a_2,1)$,
$B:(b_0,b_0\omega,b_2,1)$
spanning a line with $P$, then $P, A, B$ determine a Baer subline.
In fact, if we suppose $B=P+\alpha A$ for some $\alpha\in\GF(q^2)^*$, then this
Baer subline is $\{A\}\cup \{\langle p+t\cdot\alpha a\rangle \mid t\in\GF(q)\}$ where $A=\langle a\rangle$
and $P=\langle p\rangle$.

Let $\mathfrak{u}$ be the polarity defining $\mathsf{H}(3,q^2)$.
Since $P$ is precisely the nullspace of $U$, and the tangent line $P^{\mathfrak{u}}\cap \pi$ is contained in
the dual Baer subline, it is not difficult to calculate that $U$ can be written explicitly as

$$
U:
    \begin{pmatrix}
     -\delta\omega^q &   \delta  &           -\gamma\omega\\
    \delta^q       &    \delta^q\omega    &  \gamma\\
      -\gamma^q\omega^q &   \gamma^q     &       0\\
    \end{pmatrix},\quad \delta\omega^q=\delta^q\omega.
$$

If we also suppose that the points of $A^{\mathfrak{u}}\cap \pi$ and $B^{\mathfrak{u}}\cap \pi$ are contained
in the dual Baer subline defined by $U$, then
we can solve for $\alpha$ and $\gamma$ (but the expressions might be ugly!). Here we explore a simple example
where $A:(0,0,1,\omega)$. Now $A^\mathfrak{u}\cap\pi$ has points of the form
$(r,s,0,0)$, $\N(r)+\N(s)=0$. So if $(r,s,0,0)$ also satisfies $(r,s,0)U(r^q,s^q,0)^T=0$, then
$$(r,s,0)U(r^q,s^q,0)^T=\T(rs^q\delta)+2\N(s)\delta^q\omega.$$

So $\T(rs^q\delta)+2\N(s)\delta^q\omega=0$ for every $(r,s,0,0)$ satisfying $\N(r)+\N(s)=0$.
In particular, $\delta$ is forced to be zero. Therefore, we can write
$$
U:
    \begin{pmatrix}
     0 &   0  &           -\gamma\omega\\
    0       &    0   &  \gamma\\
      -\gamma^q\omega^q &   \gamma^q     &       0\\
    \end{pmatrix}.
$$

\subsection{Proof of the first part of Theorem \ref{thm:constU3}}\label{proofconstU3}

Here we prove that the incidence structure $\Gamma$ of Theorem \ref{thm:constU3} is a generalised
hexagon.  Our approach is to use a definition of a generalised hexagon which is equivalent to the
one stated in the introduction: (i) it contains no ordinary $k$-gon for $k\in\{2,3,4,5\}$, (ii) any
pair of elements is contained in an ordinary hexagon, and (iii) there exists an ordinary heptagon
(see \cite[\S 1.3.1]{HvM}).  A thick generalised polygon has \textit{order} $(s,t)$ if every line has $s+1$
points and every point is incident with $t+1$ lines.  A counting argument shows that if we know that
the number of points and lines of a generalised hexagon are $(s+1)(1+st+s^2t^2)$ and
$(t+1)(1+st+s^2t^2)$, then the conditions (ii) and (iii) automatically follow from the first
condition.

\begin{proof}
First we show that $\Omega$ induces a point-partition of each generator (minus its point in the
Hermitian curve $\mathcal{O}$). Let $\ell$ be a generator of $\mathsf{H}(3,q^2)$ and let $P$ be a point of
$\ell\backslash\mathcal{O}$. For a point $X$, we will let $X^*$ be the $q+1$ elements of $\Omega$
which lie on $X$.
Consider the $q+1$ elements $P^*$ of $\Omega$ on $P$. Since $P^*$
covers the points of a Baer subplane, it follows that there is a unique element of $\Omega$
contained in $\ell$ and containing $P$. Therefore $\Omega$ induces a point-partition of each
generator minus its point in the Hermitian curve $\mathcal{O}$. It follows immediately that $\Gamma$ is a partial
linear space (i.e., every two points lie on at most one line).

Since $\mathsf{H}(3,q^2)$ is a generalised quadrangle, $\Gamma$ has no triangles. So suppose now
that we have a quadrangle $R$, $S$, $T$, $U$ of $\Gamma$. Then at least three of these points
are necessarily affine points. For example, if two of these points were of type (a), two points of type (b),
and with one line of type (i) and three of type (ii) making up the quadrangle,
the three lines of type (ii) would yield a triangle of generators. So this case is clearly impossible.
At least three points, $S$, $T$, $U$
say, are necessarily affine points and the lines of the quadrangle are elements of
$\Omega$. Moreover, $R$ is also an affine point, since if $R$ were a generator then $S$ and $U$
would lie on $R$ and $ST$, $TU$, $SU$ would then be a triangle in $\mathsf{H}(3,q^2)$; a
contradiction. So all the four points $R$, $S$, $T$, $U$ of a quadrangle must be affine.

Recall that $\mathfrak{u}$ is the polarity defining $\mathsf{H}(3,q^2)$. Note that
$R^\mathfrak{u} \cap T^\mathfrak{u}$ is equal to $SU$ and that $SU\cap
\mathsf{H}(3,q^2)$ is a Baer subline with a point on $\mathcal{O}$.  Indeed $R^*$
spans a Baer subplane fully contained in $\mathsf{H}(3,q^2)$ and it meets $\mathcal{O}$
 in a Baer subline and since $R^\mathfrak{u}\cap T^\mathfrak{u}\cap
\mathsf{H}(3,q^2)$ is a Baer subline contained in $R^*$ then $SU\cap \mathsf{H}(3,q^2)$
has a point in $\mathcal{O}$.
Likewise $S^\mathfrak{u}\cap U^\mathfrak{u}$ equal to $RT$ and $RT\cap
\mathsf{H}(3,q^2)$ is a Baer subline with a point in $\mathcal{O}$.  So $SU$ and
$RT$ are polar to each other under $\mathfrak{u}$, but then each point
of $\mathsf{H}(3,q^2)$ on $SU$ is collinear with each point of $\mathsf{H}(3,q^2)$ on
$RT$, while the points of $\mathcal{O}$ are pairwise non-collinear, a
contradiction. Hence $\Gamma$ has no quadrangles.

Suppose we have a pentagon $R$, $S$, $T$, $U$, $W$ of $\Gamma$.
Now points of type (b), which are affine points, are collinear in $\Gamma$ if they are
incident with a common element of $\Omega$. Since each element of $\Omega$ spans a generator, points
of type (b) are also collinear in $\mathsf{H}(3,q^2)$. So since $\mathsf{H}(3,q^2)$ is a generalised
quadrangle, we see immediately that each point of our pentagon is an affine point.
Suppose, by way of contradiction, that
our pentagon has a point of type (a), that is, a generator $\ell$ of $\mathsf{H}(3,q^2)$.
Then we would have four generators of $\mathsf{H}(3,q^2)$ forming a quadrangle and we
obtain a similar ``forbidden'' quadrangle of affine points (i.e., $RSTU$) from the above argument.
So there are no
pentagons in $\Gamma$.

A trivial counting argument shows that $\mathcal{L}$ has size $(q^6-1)/(q-1)$, which is equal to the
sum of the number of affine points and the number of generators of $\mathsf{H}(3,q^2)$, and so it
follows that $\Gamma$ is a generalised hexagon (of order $(q,q)$).  \end{proof}

%
%

\subsection{Exhibiting a suitable set of Baer subgenerators}\label{sect:partition}

In this section, we describe a natural candidate for a set $\Omega$ of Baer subgenerators satisfying
the hypotheses of Theorem \ref{thm:constU3}.  Consider the stabiliser $G_\mathcal{O}$ in
$\mathsf{PGU}_4(q)$ of the Hermitian curve $\mathcal{O}=\pi \cap \mathsf{H}(3,q^2)$, where $\pi$ consists of the
elements whose last coordinate is zero. Then the elements of $G_\mathcal{O}$ can be thought of
(projectively) as matrices $M_A$ of the form
$$M_A:=\left(\begin{smallmatrix}
&&&0\\
&A&&0\\
&&&0\\
0&0&0&1
\end{smallmatrix}\right), \quad A\in\mathsf{GU}_3(q).$$

\begin{lemma}\label{lemma:transBaer}
The group $G_{\mathcal{O}}$ acts transitively on the set of Baer subgenerators which
have a point in $\mathcal{O}$.
\end{lemma}

\begin{proof}
Inside the group $\mathsf{PGU}_4(q)$, the stabiliser $J$ of a generator $\ell$ induces a $\mathsf{PGL}_2(q^2)$ acting $3$-transitively
on the points of $\ell$. So the stabiliser in $J$ of a point $P$ of $\ell$ acts transitively on the Baer sublines within $\ell$ which contain $P$.
Now $J$ meets $G_\mathcal{O}$ in the stabiliser of a point of $\ell$, and so $G_{\mathcal{O},\ell}$ acts transitively on Baer subgenerators
contained in $\ell$. Since $G_\mathcal{O}$ acts transitively on $\mathcal{O}$, the result follows.
\end{proof}

The key to this construction is the action of a particular subgroup of $G_{\mathcal{O}}$. We will
see later that this group naturally corresponds to the stabiliser in $\mathsf{G}_2(q)$ of a
non-degenerate hyperplane $\Q^-(5,q)$ of $\Q(6,q)$.

\begin{defn}[$\mathsf{SU}_3$]
Let $\mathsf{SU}_3$ be the group of collineations of $\mathsf{H}(3,q^2)$ obtained from the
matrices $M_A$ where $A\in\mathsf{SU}_3(q)$.
\end{defn}

In short, the orbits of $\mathsf{SU}_3$ on Baer subgenerators with a point in $\mathcal{O}$, each
form a suitable candidate for $\Omega$, as we will see.

\begin{lemma}\label{lem:stabbaer}
Let $\mathcal{O}=\pi\cap \mathsf{H}(3,q^2)$, where $\pi$ is the hyperplane $X_3=0$ of $\PG(3,q^2)$,
and let $G_\mathcal{O}$ be the stabiliser of $\mathcal{O}$ in $\mathsf{PGU}_4(q)$. Let $b$ be a Baer
subgenerator of $\mathsf{H}(3,q^2)$ with a point in $\mathcal{O}$.  Then the stabiliser of $b$ in
$G_{\mathcal{O}}$ is contained in $\mathsf{SU}_3$.
\end{lemma}

\begin{proof}
Recall from the beginning of Section \ref{sect:hermitian} that given a Baer subgenerator $b$ of $\mathsf{H}(3,q^2)$ with a point
$B$ in $\mathcal{O}$, there is a dual Baer subline of $\pi$ with vertex $B$. So there is a set
of $3\times 3$ Hermitian matrices $U$ of rank $2$, which are equivalent up to scalar multiplication
in $\GF(q^2)^*$. Now $G_{\mathcal{O}}$ induces an action on the pairs $[U,\ell]$, where $U$ is a
Hermitian matrix of rank $2$ and $\ell$ is a generator containing the nullspace of $U$, which we can
write out explicitly by
$$[U,\ell]^{M_A}=[A^{-1}UA,\ell^{M_A}].$$

Let $\omega$ be an element of $\GF(q^2)$ satisfying $\N(\omega)=-1$, and let $U_0$ and
$\ell_0$ be
$$U_0:=\left(\begin{smallmatrix}
0&0&-\omega\\
0&0&1\\
-\omega^q&1&0\\
\end{smallmatrix}\right), \quad
\ell_0 := \langle (1,\omega,0,0), (0,0,1,\omega)\rangle.$$ Since $G_{\mathcal{O}}$ acts transitively
on Baer subgenerators with a point in $\mathcal{O}$ (Lemma \ref{lemma:transBaer}), we need only calculate the stabiliser of $[U_0,\ell_0]$.  Now let $M_A$ be an
element of $G_ {\mathcal{O}}$ fixing $[U_0,\ell_0]$.  Since $M_A$ fixes $\ell_0$, we can see by
direct calculation that $A$ is of the form
$$\left(\begin{smallmatrix}
a&b&-f\omega\\
d&e&f\\
g&g\omega&1\\
\end{smallmatrix}\right),$$
with $(a+d\omega)\omega =b+e\omega$.\\
Now we see what it means for $A$ to centralise $U_0$ up to a scalar $k$, that is,
$U_0A=kAU_0$. Hence
$$
\left(\begin{smallmatrix}
-g \omega& -g \omega^2& -\omega\\
g& g \omega& 1\\
 d - a \omega^q& e - b \omega^q&   0
 \end{smallmatrix}\right)
  =
k\left(\begin{smallmatrix}
-f& -f \omega& b - a \omega\\
-f \omega^q& f& e - d \omega\\
-\omega^q& 1& 0
\end{smallmatrix}\right)$$
and we obtain
$$A=\left(\begin{smallmatrix}
k^{-1}-b\omega^q&b&-k^{-1}g\omega^2\\
(k^{-1}-k-b\omega^q)\omega^q&k+b\omega^q&k^{-1}g\omega\\
g&g\omega&1\\
\end{smallmatrix}\right)$$
where $k\in\GF(q)$, $\N(g)=k^2+\T(b^q\omega)-1$ and $\T(g\omega)=0$  (in order for this matrix to be unitary).  

The determinant of $A$ is
$$1-g^2\omega(\N(\omega)+1) (\omega k^{-2} + b (\N(\omega)+1)k^{-1} +\omega)=1$$
and therefore, the stabiliser of $[U_0,\ell_0]$ in $G_{\mathcal{O}}$ is contained in $\mathsf{SU}_3$.
\end{proof}

The above lemma allows us to attach a value to a Baer subgenerator that is an invariant
for the action of $\mathsf{SU}_3$.

\begin{defn}[Norm of a Baer subgenerator]\label{norm}
Let $\mathcal{O}$ be the Hermitian curve $\mathsf{H}(3,q^2)\cap \pi$, where $\pi$ is the hyperplane
$X_3=0$ of $\PG(3,q^2)$ and let $G_\mathcal{O}$ be the stabiliser of $\mathcal{O}$ in
$\mathsf{PGU}_4(q)$.  Fix a Baer subgenerator $b_0$ of $\mathsf{H}(3,q^2)$ with a point in
$\mathcal{O}$.  Let $b$ be a Baer subgenerator of $\mathsf{H}(3,q^2)$ with a point in $\mathcal{O}$,
and suppose $M_A$ is an element of $G_\mathcal{O}$ such that $b=b_0^{M_A}$.  Then the \textit{norm}
of $b$ is
$$\Vert b\Vert := \det(A).$$ Moreover (by Lemma \ref{lem:stabbaer}), the map $b\mapsto \Vert b\Vert$ induces a group
homomorphism $\phi$ from $G_{\mathcal{O}}$ to the multiplicative subgroup of elements of
$\GF(q^2)^*$ satisfying $\N(x)=1$.
\end{defn}

Note that the kernel of $\phi$ is $\mathsf{SU}_3$. The homomorphism $\phi$ is surjective and hence
there is a natural partition of Baer subgenerators with a point in $\mathcal{O}$ into $q+1$
classes. Each orbit of $\mathsf{SU}_3$ consists of Baer subgenerators with a common value for their norm.

\begin{lemma}\label{lemma:omega}
Let $\mu$ be an element of $\GF(q^2)$ such that $\N(\mu)=1$.  Let $\mathcal{O}$ be a Hermitian
curve of $\mathsf{H}(3,q^2)$ defined by $X_3=0$, and let $\Omega$ be a set of Baer subgenerators
with a point in $\mathcal{O}$ which have norm equal to $\mu$. Then:
\begin{enumerate}
\item[(i)] Every affine point is on $q+1$ elements of $\Omega$ covering a Baer subplane.
\item[(ii)] For every point $X\in\mathcal{O}$ and for every affine point $Y$ in
  $X^\mathfrak{u}$, there is a unique element of $\Omega$ through $X$ and $Y$.
\end{enumerate}
\end{lemma}

\begin{proof}
Recall that $\Omega$ is an orbit of $\mathsf{SU}_3$ on Baer subgenerators and $\mathsf{SU}_3$ acts
transitively on the affine points $\mathsf{H}(3,q^2)\backslash\mathcal{O}$, and so clearly every
affine point is on $q+1$ elements of $\Omega$. Moreover, such a set of $q+1$ elements of $\Omega$
will cover a Baer subplane, as we show now. Let $Y$ be an affine point, let $Y^*$ be the set
of $q+1$ elements of $\Omega$ through $Y$ and let $b_0$ be one particular
element of $Y^*$. Then every other element of $Y^*$ is in the orbit of $b_0$ under the stabiliser of $Y$
in $\mathsf{SU}_3$. Now for every $g\in (\mathsf{SU}_3)_Y$, we know that $\langle b_0^g\rangle=\langle b_0\rangle^g
\in Y^\perp$ and so every element of $Y^*$ lies in the plane $Y^\perp$. At infinity, $Y^\perp$ meets $\mathcal{O}$
in a Baer subline and so we have a triangle of Baer sublines spanning a Baer subplane of $Y^\perp$, and
it is covered completely by the elements of $Y^*$.

To complete the proof, we need only prove (ii).
Since the stabiliser of a point in $\mathcal{O}$ is transitive on the set of affine points in the perp of that point, we can
assume that $X=(1,\omega,0,0)$ and $Y=(0,0,1,\omega)$ for some $\omega$ satisfying
$\N(\omega)=-10$.  We have already seen, in the proof of Lemma \ref{lem:stabbaer}, that $X$ and
$Y$ lie on a Baer subgenerator, which we can assume without loss of generality, is in $\Omega$.
This Baer subgenerator is uniquely defined by a $3\times 3$ Hermitian matrix $U$ of rank $2$ and the
generator $\ell$ spanning $X$ and $Y$, and we assume (as before) that $U$ has the form
$$U:=\left(\begin{smallmatrix}
0&0&-\omega\\
0&0&1\\
-\omega^q&1&0\\
\end{smallmatrix}\right).$$
Then the two-point stabiliser of $X$ and $Y$ inside $\mathsf{SU}_3$ consists of elements $M_A$ with $A$
of the form
$$A=\left(\begin{smallmatrix}
a&b&0\\
d&e&0\\
0&0&1\\
\end{smallmatrix}\right)$$
where $(a+d\omega)\omega=b+e\omega$ and
$\left(\begin{smallmatrix}
a&b\\
d&e
\end{smallmatrix}\right)\left(\begin{smallmatrix}
a^q&d^q\\
b^q&e^q
\end{smallmatrix}\right)=I$. Let's consider one of these elements $M_A$. Then
\begin{align*}
(A^q)^TUA
&=\left(\begin{smallmatrix}
0&0&-a^q\omega+d^q\\
0&0&-b^q\omega+e^q\\
-a\omega^q+d&-b\omega^q+e&0\\
\end{smallmatrix}\right)
\end{align*}
and we see that this matrix is a scalar multiple of $U$ (the scalar being $(-b\omega^q+e)$).
Therefore $M_A$ fixes the Baer subgenerator defined by $[U,\ell]$. Hence there is a unique element
of $\Omega$ on $X$ and $Y$.  \end{proof}

%
%

\subsection{Classifying the suitable sets of Baer subgenerators}\label{sect:classify}

\begin{theorem}\label{thm:omega}
Suppose $\Omega$ is a set of Baer subgenerators of $\mathsf{H}(3,q^2)$ with a point in $\mathcal{O}$,
such that every affine point is on $q+1$ elements of $\Omega$ spanning a Baer subplane.  Then
$\Omega$ is an orbit under $\mathsf{SU}_3$.
\end{theorem}

\begin{proof}
Let $b$ be a Baer subgenerator of $\mathsf{H}(3,q^2)$ with a point in $\mathcal{O}$. If $b'$ is
another Baer subgenerator of $\mathsf{H}(3,q^2)$ with a point in $\mathcal{O}$ such that $b$ and
$b'$ meet in an affine point and span a fully contained Baer subplane, then we will show that there
is some element of $\mathsf{SU}_3$ which maps $b$ to $b'$.  Without loss of generality, we can
choose our favourite Baer subgenerator and our favourite affine point.  Suppose we have a fixed Baer
subgenerator $b$ giving the dual Baer subline defined by
$$U=\left(\begin{smallmatrix}
0&0&-\omega&\\
0&0&1\\
-\omega^q&1&0\\
\end{smallmatrix}\right)$$
and on the generator $\ell=\langle (1,\omega,0,0),(0,0,1,\omega)\rangle$ where $\N(\omega)=-1$.
Let $P$ be the affine point $(0,0,1,\omega)$ and consider an arbitrary generator $\ell'$ on $P$
where $\ell':=\langle (0,0,1,\omega),(1,\nu,0,0)\rangle$ and $\N(\nu)=-1$. Suppose we have a Baer
subgenerator $b'$ on $P$, on the generator $\ell'$, defined by the matrix $U'$. Since every element
of $P^\mathfrak{u}\cap\mathcal{O}$ is in the dual Baer subline defined by $U'$, we have that $U'$ can be written as
$$\left(\begin{smallmatrix}
a&0&\beta\\
0&a&\gamma\\
\beta^q&\gamma^q&c\\
\end{smallmatrix}\right)$$
where $a\in\GF(q)$ and $\beta,\gamma\in\GF(q^2)$. For $(1,\nu,0,0)$ to be in the nullspace of
$U'$, we must have $a=0$ and $\beta=-\gamma\nu$. That is, $U'$ is just
$$\left(\begin{smallmatrix}
0&0&-\gamma\nu\\
0&0&\gamma\\
-\gamma^q\nu^q&\gamma^q&c\\
\end{smallmatrix}\right).$$

Now $b$ and $b'$ span a fully contained Baer subplane if and only if the dual Baer sublines defined by $U$ and $U'$ share
 only the points of $P^\mathfrak{u} \cap \mathcal{O}$, on $\mathcal{O}$.
Indeed suppose, by way of contradiction, that  there is a point $Z$ of
$\mathcal{O}$ in common between the dual
Baer sublines defined by $U$ and $U'$. Then $Z^\mathfrak{u}$ meets $b$
in a point $Q$, different from $L$ and $P$ and  it meets $b'$ in a point
$Q'$ different from $L'$ ($L'=\pi\cap b'$) and $P$.
Thus $Z^\mathfrak{u} \cap P^\mathfrak{u}$ meets $\mathsf{H}(3,q^2)$ in a Baer
subline $b''$ containing $Q$ and $Q'$.  Now the Baer subplane spanned
by $b$ and $b'$ is fully contained if and only if  $b''$ has a point
$T$ in $\mathcal{O}$.
This implies that $T$ and $Z$ are points of $\mathcal{O}$ collinear on
$\mathsf{H}(3,q^2)$; a contradiction.

 Note that $P^\mathfrak{u}\cap\mathcal{O}$ consists of the points of the form $(1,\delta,0,0)$ together with $(0,1,0,0)$.
  Suppose $(1,\delta,\eta,0)$ is an element of both dual
Baer sublines. That is, $(1,\delta,\eta)U(1,\delta^q,\eta^q)^T=0$ and
$(1,\delta,\eta)U'(1,\delta^q,\eta^q)^T=0$. Now

\begin{align*}
(1,\delta,\eta)U(1,\delta^q,\eta^q)^T&= (1,\delta,\eta)\left(\begin{smallmatrix}
0&0&-\omega&\\
0&0&1\\
-\omega^q&1&0\\
\end{smallmatrix}\right)(1,\delta^q,\eta^q)^T\\
&=-\eta\omega^q+\eta\delta^q+(-\omega+\delta)\eta^q\\
&=T(\eta(\delta-\omega)^q),
\end{align*}
\begin{align*}
(1,\delta,\eta)U'(1,\delta^q,\eta^q)^T&= (1,\delta,\eta)\left(\begin{smallmatrix}
0&0&-\gamma\nu\\
0&0&\gamma\\
-\gamma^q\nu^q&\gamma^q&c\\
\end{smallmatrix}\right)(1,\delta^q,\eta^q)^T\\
&=-\eta\gamma^q\nu^q+\eta\gamma^q\delta^q+(-\gamma\nu+\delta\gamma+\eta c)\eta^q\\
&=-(\eta\gamma^q\nu^q+\eta^q\gamma\nu)+(\eta\gamma^q\delta^q+\eta^q\gamma\delta)+c\eta^{q+1}\\
&=T(\eta\gamma^q(\delta-\nu)^q)+c\N(\eta).
\end{align*}
Since $1+\N(\delta)+\N(\eta)=0$, we see that our equations become
\begin{equation*}\tag{*}\label{traceequation}
\T(\eta(\delta-\omega)^q)=0\text{ and }\T(\eta\gamma^q(\delta-\nu)^q)=c(1+\N(\delta))
\end{equation*}
So since the dual Baer sublines defined by $U$ and $U'$ share
only the points of $P^\mathfrak{u}\cap\mathcal{O}$, then whenever condition (\ref{traceequation}) holds for a choice of $\delta$, $\eta$,
we will have $\eta=0$. Therefore, we must have a priori that $c=0$ and $\gamma\notin\GF(q)$.

Let $\eta=(\gamma\nu-\gamma^q\omega)^q$ and
$$\delta=\frac{-\eta^q+\eta^{q-1}\gamma^q(\omega-\nu)^q}{\gamma^q-\gamma}.$$
Then a straightforward
calculation shows that $1+\N(\delta)+\N(\eta)=0$, $\T(\eta(\delta-\omega))= 0$ and $\T(\eta\gamma(\delta-\nu))= 0$, so condition (\ref{traceequation}) holds, and hence
$\eta=0$.
Therefore, $\nu = \omega\gamma^{q-1}$ and
$$U'=\left(\begin{smallmatrix}
0&0&-\gamma^q\omega\\
0&0&\gamma\\
-\gamma \omega^q&\gamma^q&0\\
\end{smallmatrix}\right).$$

We want to show that $U'$ is conjugate to $U$ under some element of $\mathsf{SU}_3(q)$.  Now the
group $\mathsf{SU}_2(q)$ of invertible $2\times 2$ matrices with unit determinant, and fixing the
form $X_0Y_0^q+X_1Y_1^q=0$ on $\GF(q^2)^2$, has $q+1$ orbits on totally isotropic vectors of
$\GF(q^2)^2$. Each orbit consists of vectors $(x,y)$ where $y/x^q$ attains a common
value. Therefore, there exists some element $C_0$ of $\mathsf{SU}_2(q)$ such that
$C_0(-\gamma\nu,\gamma)^T=(-\omega,1)$.  Let
$${\small C:=\left(\begin{array}{c|c}
C_0 &\begin{smallmatrix}0\\0\\ \end{smallmatrix} \\
\hline
 \begin{matrix}0&0 \end{matrix}& 1\\
\end{array}\right).}$$
Then one can check easily that $C$ has determinant $1$ and $CU(C^q)^T=U'$.  Therefore, there is some
element of $\mathsf{SU}_3$ which maps $b$ to $b'$.

For every affine point $P$, let $P^*$ be the set of $q+1$ elements of $\Omega$ incident with
$P$. Then by the above, every element of $P^*$ is contained in a common orbit of $\mathsf{SU}_3$.
Note that $\mathsf{SU}_3$ is transitive on generators of $\mathsf{H}(3,q^2)$, and the stabiliser of
a point $X$ of $\mathcal{O}$ in $\mathsf{SU}_3$ is transitive on the affine points of
$X^\mathfrak{u}$.  Suppose now that $b$ and $b'$ do not meet in an affine point.  Let $P\in b$. Then
$P^*\subset b^{\mathsf{SU}_3}$.  Now there exists $g\in\mathsf{SU}_3$ such that $\langle
b\rangle^g=\langle b'\rangle$ and $P^g\in b'$.  Thus $b'\in (P^g)^*\subset
(b')^{\mathsf{SU}_3}$. Note also that $P^g\in b^g$, and hence $b^g\in
(b')^{\mathsf{SU}_3}$. Therefore $b$ and $b'$ are in the same orbit under $\mathsf{SU}_3$.
\end{proof}

In Section \ref{sect:proofQ6embedding}, we will use the above result to prove Theorem
\ref{thm:Q6embedding}.

%
%

\section{The connection with the 6-dimensional parabolic quadric}\label{sect:parabolicquadric}

A non-degenerate hyperplane section of $\Q(6,q)$ can be of one of two types (up to isometry):
it could induce a hyperbolic quadric $\Q^+(5,q)$ or it could induce an elliptic
quadric $\Q^-(5,q)$.  The stabiliser of a hyperbolic quadric section in $\mathsf{G}_2(q)$ is
isomorphic to $\mathsf{SL}_3(q):2$, whilst the stabiliser of an elliptic quadric section in
$\mathsf{G}_2(q)$ is isomorphic to $\mathsf{SU}_3(q):2$ (see \cite{Kle88a}). These two maximal
subgroups bring forth the two low-dimensional models of the Split Cayley hexagon that appear in this paper,
and a second way to explain the interplay between these `linear' and `unitary' models is 
via Curtis-Tits and Phan systems; see Section \ref{sect:phan}.
We begin first with some
observations about the situation where we fix a $\Q^+(5,q)$ hyperplane section.

The stabiliser $\mathsf{SL}_3(q):2$ of $\Q^+(5,q)$ fixes two disjoint planes $p'$ and $\sigma'$ of
$\Q^+(5,q)$, and then the lines of $\mathcal{H}(q)$ contained in $\Q^+(5,q)$ are just the lines of
$\Q^+(5,q)$ which meet both $p'$ and $\sigma'$ in a point.  It was noticed in \cite{CameronKantor79}
that we can reconstruct $\mathcal{H}(q)$ from these two fixed planes together with an orbit $\Omega$
of $\mathsf{SL}_3(q)$ on affine lines (of size $(q^3-q)(q^2+q+1)$).  We can capture the affine
points by noticing that the $q+1$ hexagon-lines through an affine point span a totally isotropic plane
(sometimes known as an $\mathcal{H}(q)$-plane) meeting $\Q^+(5,q)$ in a line disjoint from both $p'$
and $\sigma'$. Similarly, we can take the polar image of an affine line and consider its
intersection with $\Q^+(5,q)$. This results in a $3$-dimensional quadratic cone of $\Q^+(5,q)$
meeting both $p'$ and $\sigma'$ in a point, but having vertex not in $p'$ nor $\sigma'$.  We can
then employ the Klein correspondence to map our projection of $\mathcal{H}(q)$ on $\Q^+(5,q)$, to
$\mathsf{PG}(3,q)$ (see \cite[\S 15.4]{Hirschfeld85} for more on the Klein correspondence). We summarise this
correspondence below:

\begin{center}
\begin{table}[H]\footnotesize
\begin{tabular}{p{8cm}|p{8cm}}
$\PG(3,q)$&$\Q(6,q)$\\ \hline
\rowcolor[gray]{0.95}Point-plane anti-flag $(p,\sigma)$ &A latin $p'$ and greek plane $\sigma'$ defining a hyperbolic quadric $\Q^+(5,q)$\\
Pencils with vertex not in $\sigma$ and plane not through $p$&Affine points of $\Q(6,q)\backslash \Q^+(5,q)$\\
\rowcolor[gray]{0.95}Lines& Points of $\Q^+(5,q)$\\
Pencils with vertex in $\sigma$ and plane through $p$& Lines of $\Q^+(5,q)$ meeting $p'$ and $\sigma'$ in a point\\
\rowcolor[gray]{0.95}Parabolic congruences & Affine lines of $\Q(6,q)$, quadratic cones of $\Q^+(5,q)$\\
Parabolic congruences having axis not incident with $p$ or $\sigma$, but having a pencil of lines with one line incident with $p$
and another incident with $\sigma$ & Affine lines of $\mathcal{H}(q)$\\
\hline
\end{tabular}
\medskip
\caption{The extended Klein representation.}
\end{table}
\end{center}

Now we describe how we can view $\mathcal{H}(q)$ as substructures of the $3$-dimensional Hermitian
surface $\mathsf{H}(3,q^2)$.  A \textit{$t$-spread} of $\PG(d,q)$ is a collection of $t$-dimensional
subspaces which partition the points of $\PG(d,q)$. So necessarily, $t+1$ must divide $d+1$ and the
size of a $t$-spread of $\PG(d,q)$ is $(q^{d+1}-1)/(q^{t+1}-1)$.  If $t+1$ is half of $d+1$, we
usually call a $t$-spread just a \textit{spread} of $\PG(d,q)$. Suppose we have a $t$-spread
$\mathcal{S}$ of $\PG(d,q)$ and embed $\PG(d,q)$ as a hyperplane in $\PG(d+1,q)$. If we define the
\textit{blocks} to be the $(t+1)$-dimensional subspaces of $\PG(d+1,q)$ not contained in $\PG(d,q)$
incident with an element of the $t$-spread, then together with the affine points
$\PG(d+1,q)\backslash\PG(d,q)$, we obtain a linear space; in fact, a $2$--$(q^{d+1},q^{t+1},1)$
design. This linear representation of a $t$-spread is a generalisation of the commonly called
\textit{Andr\'e/Bruck-Bose construction} (where $t+1=(d+1)/2$), and is fully explained by Barlotti
and Cofman \cite{BaCo}. More generally, it is possible that this construction produces a
Desarguesian affine space and we then say that the given $t$-spread is \textit{Desarguesian}. It
turns out that a $t$-spread $\mathcal{S}$ is Desarguesian if and only if $\mathcal{S}$ induces a
spread in any subspace generated by two distinct elements of $\mathcal{S}$ (see \cite{Lun} and
\cite{Segre64}).

Now consider $\PG(3,q^2)$ and a hyperplane $\pi_\infty$ therein, and identify $\mathsf{AG}(3,q^2)$
with the affine geometry $\PG(3,q^2)\backslash \pi_\infty$.
We will be considering the correspondence between objects in $\mathsf{H}(3,q^2)$ and $\Q(6,q)$,
where $\mathcal{S}$ is a Hermitian spread of a non-degenerate hyperplane section $\Q^-(5,q)$ of
$\Q(6,q)$. One can also obtain this correspondence via field reduction from $\mathsf{H}(3,q^2)$ to
$\Q^+(7,q)$, and then slicing with a non-degenerate hyperplane section (see \cite{Lunardon06}).  We
will call this correspondence the \textit{Barlotti-Cofman-Segre representation} of
$\mathsf{H}(3,q^2)$.  Below we summarise the various correspondences between objects in
$\mathsf{H}(3,q^2)$ and objects in $\Q(6,q)$ obtained by the Barlotti-Cofman-Segre representation of
$\mathsf{H}(3,q^2)$. Throughout, we fix a hyperplane $\Sigma_\infty$ at infinity intersecting
$\Q(6,q)$ in a $\Q^-(5,q)$, which corresponds to a fixed non-degenerate hyperplane $\pi_\infty$ of
$\mathsf{H}(3,q^2)$, and we let $\mathcal{S}$ denote a Hermitian spread of $\Sigma_\infty$.

\begin{center}
\begin{table}[H]\footnotesize
\begin{tabular}{p{8cm}|p{8cm}}
$\mathsf{H}(3,q^2)$&$\Q(6,q)$\\ \hline
\rowcolor[gray]{0.95}Hermitian curve $\mathcal{O}$ of $\pi_\infty$ & Hermitian spread $\mathcal{S}$ of $\Q^-(5,q)$\\
  Affine points $\mathsf{H}(3,q^2)\backslash \pi_\infty$ & Affine points of $\Q(6,q)\backslash  \Q^-(5,q)$\\
\rowcolor[gray]{0.95}Generators of $\mathsf{H}(3,q^2)$&Generators of $\Q(6,q)$ incident with some element of $\mathcal{S}$\\
  Baer subplane contained in  $\mathsf{H}(3,q^2)$ meeting $\mathcal{O}$ in a Baer subline&Generators of  $\Q(6,q)$ not incident with any element of $\mathcal{S}$\\
\rowcolor[gray]{0.95}Baer subgenerators with a point in $\mathcal{O}$& Affine lines of $\Q(6,q)$\\ \hline
\end{tabular}
\medskip
\caption{The Barlotti-Cofman-Segre representation.}\label{BCSrepresentation}
\end{table}
\end{center}

The table below shows how we can directly obtain the model for the split Cayley hexagon on the
$3$-dimensional Hermitian surface via the Barlotti-Cofman-Segre correspondence. We can recover the
affine points of $\Q(6,q)$ by noticing that a plane incident with a spread element will correspond
to a hexagon-plane; a point of $\mathcal{H}(q)$ together with its $q+1$ incident lines.

\begin{center}
\begin{table}[H]\footnotesize
\begin{tabular}{ll|l}
&In $\mathsf{H}(3,q^2)$& Barlotti-Cofman-Segre image in $\Q(6,q)$ \\
\hline
\rowcolor[gray]{0.95} \textsc{Points} & (a) Lines of $\mathsf{H}(3,q^2)$ & Planes of $\Q(6,q)$ containing a spread element.\\
& (b)  Affine points of $\mathsf{H}(3,q^2)\backslash\mathcal{O}$& Affine points of $\Q(6,q)\backslash\Q^-(5,q)$.\\
\hline
\rowcolor[gray]{0.95} \textsc{Lines} & (i) Points of $\mathcal{O}$ & Lines of the Hermitian spread.\\
 & (ii) Elements of $\Omega$ & Affine lines spanning a totally isotropic plane with a spread element.\\ \hline
\end{tabular}
\caption{The split Cayley hexagon in $\mathsf{H}(3,q^2)$.}
\end{table}
\end{center}

%
%

\section{Characterising the split Cayley hexagon in the 6-dimensional parabolic quadric}\label{sect:hexagon}

By the Barlotti-Cofman-Segre correspondence, we can translate Theorem \ref{thm:constU3} to a
statement about substructures of $\Q(6,q)$. However, the information that can be transferred via
this correspondence is not sufficient to characterise a set of lines $\Q(6,q)$ as the lines of a
generalised hexagon; there is an additional case.  The natural model of the split Cayley hexagon was
revised in the introduction, and here we briefly point out a characterisation of it as a set of
lines of $\Q(6,q)$. It is a special case of a result of Cuypers and Steinbach \cite[Theorem
  1.1]{CuypersSteinbach}, but we give a direct proof for completeness.

\begin{theorem}\label{thm:HqQuadric}
Let $\mathcal{L}$ be a set of lines of $\Q(6,q)$ such that every point of $\Q(6,q)$
is incident with $q+1$ lines of $\mathcal{L}$ spanning a plane.  Then one of the following
occurs:
\begin{enumerate}
\item[(a)] There is a spread $\mathcal{S}$ of $\Q(6,q)$ such that
$\mathcal{L}$ is equal to the union of the lines contained in each generator of
$\mathcal{S}$.
\item[(b)]
The points of $\Q(6,q)$ together with $\mathcal{L}$ define the points and
lines of a generalised hexagon, and a plane of $\Q(6,q)$ contains $0$ or $q+1$ elements of
$\mathcal{L}$ in it.
\end{enumerate}
\end{theorem}

\begin{proof}
Let $\Gamma$ be the geometry having the points of $\Q(6,q)$ as its points, and having $\mathcal{L}$
as its set of lines.
Clearly $\Gamma$ is a partial linear space where there are $q+1$ lines through every point, and $q+1$
points through every line. We will write $P^*$ for the pencil of $q+1$ lines of $\mathcal{L}$ incident
with $P$.

Since every plane of $\PG(6,q)$ meets $\Q(6,q)$ in a full plane, a conic, a line, a pair of
concurrent lines or a point, it follows that every plane intersects $\mathcal{L}$ in $q^2+q+1$,
$q+1$, $1$ or $0$ lines. We will show now that the first possibility leads to case (a). Suppose
there is a plane $\pi$ with $q^2+q+1$ lines of $\mathcal{L}$.  Let $\ell$ be an element of
$\mathcal{L}$ not contained in such a plane. Then the $q+1$ planes on the tangent quadric containing $\ell$ (i.e., the points collinear to all the points on $\ell$)
contain $q+1$ elements of $\mathcal{L}$. Since there is always at least one point $p$ of $\pi$
collinear with all points of $\ell$, we see that the point $p$ is now incident with at least $q+2$
elements of $\mathcal{L}$; a contradiction. Hence either every point is in a plane with $q^2+q+1$
elements of $\mathcal {L}$ (and we obtain the spread of $\Q(6,q)$), or no point is.

Suppose now that $\mathcal{L}$ is not partitioned by a spread of $\Q(6,q)$.  So no plane of
$\PG(6,q)$ contains $q^2+q+1$ elements of $\mathcal{L}$, and therefore, every plane intersects
$\mathcal{L}$ in $0$, $1$ or $q+1$ lines. We continue now to prove that $\Gamma$ is a generalised
hexagon. Clearly there is no triangle formed by lines of $\mathcal{L}$, so suppose we have a
quadrangle $R$, $S$, $T$, $U$ in $\Gamma$. Note that these points do not lie in a common plane.  The
planes spanning $T^*$, $U^*$ and $R^*$ are three totally singular planes contained in a common
3-space, which implies that this 3-space is also totally singular; a contradiction.  Suppose now we
have a pentagon $R$, $S$, $T$, $U$, $W$ of $\Gamma$, (and the ordering of these points is
important).  So $RSTU$ spans a 3-space intersecting $\Q(6,q)$ in two totally singular planes, namely
$S^*$ and $T^*$. Now $W$ is collinear (in $\mathcal{L}$) with $R$ and $U$, and therefore the line
$RU$ is totally singular; which implies that $RSTU$ is totally singular, a contradiction. So there
are no $k$-gons in $\Gamma$ with $k<6$. Since $\mathcal{L}$ has size equal to the number of points
of $\Q(6,q)$, it follows that $\Gamma$ is a generalised hexagon of order $q$.

Let $N_i$ be the number of planes of $\Q(6,q)$ containing $i$ elements of $\mathcal{L}$.
So $N_0+N_1+N_{q+1}=(q+1)(q^2+1)(q^3+1)$. Now each point is on a unique
plane containing $q+1$ elements of $\mathcal{L}$, and so $N_{q+1}=(q^3+1)(q^2+q+1)$.
Now for a given point $X$, all but one of the planes on $X$ would have no lines of $\mathcal{L}$ on it,
which accounts for $N_0=q^3(q^3+1)$ planes (n.b., there are $(q+1)(q^2+1)$ planes
on any point, and a plane contains $q^2+q+1$ points). So it follows that $N_1=0$.
\end{proof}

\begin{lemma}\label{lemma:hermitianspread}
Let $\mathcal{L}$ be a set of lines of $\Q(6,q)$ such that every point $X$ of $\Q(6,q)$ is incident with
$q+1$ lines of $\mathcal{L}$ spanning a plane $X^*$, and such that the concurrency graph of $\mathcal{L}$
is connected.  Suppose $\Pi$ is a nondegenerate hyperplane meeting $\Q(6,q)$ in a $\Q^-(5,q)$-quadric.
Then the set $\mathcal{S}:=\{X^*\cap \Pi: X\in\Q^-(5,q)\}$ defines a Hermitian spread of $\Q^-(5,q)$.
\end{lemma}

\begin{proof}
Any pair of lines of $\mathcal{S}$ are disjoint since otherwise they would intersect
in a point $P$ and the plane $P^*$ spanned by the $q+1$ elements of $\mathcal{L}$ incident
with $P$ would then be contained in $\Q^-(5,q)$. Therefore, $\mathcal{S}$ forms a spread of
$\Q^-(5,q)$.

Now consider two elements $\ell$ and $m$ of $\mathcal{S}$. The solid $\langle \ell,m\rangle$ meets
$\Q^-(5,q)$ in  a $\Q^+(3,q)$ section. The polar image of $\langle \ell,m\rangle$, within $\Q(6,q)$, is then a 
plane meeting $\Q(6,q)$  in a non-degenerate conic $\mathcal{C}$.
Let $r$ be a line in the regulus determined by $\ell$ and $m$, and suppose for a proof by contradiction
that $r$ is not an element of $\mathcal{S}$. Then each of the $q+1$ points $Z_i$ on $r$ defines a different element
$\ell_i:=Z_i^*\cap \Pi$ of $\mathcal{S}$. Since the lines contained in $\langle \ell,m\rangle$ concurrent with $r$
form the opposite-regulus to that defined by $\ell$ and $m$, it follows that none of the $\ell_i$ are contained
in $\langle \ell, m \rangle$.

Since $\ell$ is a line of $\mathcal{L}$  and, by Theorem 4.1, a plane of $\Q(6,q)$ has $0$ or $q+1$ elements of
$\mathcal{L}$ contained in it, each of the $q+1$ planes $\langle \ell,X_i\rangle$ for
$X_i\in \mathcal{C}$ is a plane $Y^*$ for some $Y\in \ell$. Similarly, each of the $q+1$ planes $\langle m,X_i\rangle$
is a plane $Y^*$ for some $Y\in m$. Hence for each $X_i\in \mathcal{C}$ the plane $X_i^*$ meets $\langle l,m\rangle$ in a line of the opposite regulus. Therefore, there is a one-to-one correspondence between points 
$X_i$ of $\mathcal{C}$ and points $Z_i$ of $r$. That is, the line $X_iZ_i$ is a line of $\mathcal{L}$ for every $i$.

Recall that the concurrency graph of $\mathcal{L}$ is connected, and so by Theorem \ref{thm:HqQuadric},
$\mathcal{L}$ forms the lines of a generalised hexagon. Let $Z_1$ and $Z_2$ be two  elements on $r$.
 Now $Z_1^\perp$ is a hyperplane and $Z_2^*$ is a plane, so we have two cases: (i) $Z_2^*$ is contained in $Z_1^\perp$, or (ii) $Z_2^*$ meets $Z_1^\perp$ in a line $n$.
The first case cannot arise as a plane of $\Q(6,q)$ contained in $Z_1^\perp$ must go through $Z_1$, and we have assumed that $r$ is not in $\mathcal{L}$.
 Suppose we have the second case. Since $Z_2$ lies in $Z_1^\perp$, the line $n$ lies on $Z_2$ and so $n$ is a line of $\mathcal{L}$ in $Z_2^*$. Note that $\langle Z_1,n\rangle$ is a plane of $\Q(6,q)$ having
at least one element of $\mathcal{L}$ contained in it.
By Theorem 4.1, a plane of $\Q(6,q)$ has $0$ or $q+1$ elements of $\mathcal{L}$ contained
in it. Therefore, there is some point $V$ on $n$ such that $Z_1\in V^*$.
Hence we have a line of $\mathcal{L}$ going through $Z_1$ concurrent with $n$,
and $Z_1$ and $Z_2$ are at distance $4$.
This requirement then forces $r$ to lie in $V^*$.
 and hence each $Z_i^*$ goes through $V$.

Now $\langle \mathcal{C}\rangle $ is a non-degenerate plane through $\Pi^\perp$ and
$\Pi^\perp \notin \langle Z_i^*: Z_i\in r\rangle$. Therefore, each $Z_i^*$ meets
the conic only in the point $X_i$.
The lines $X_1Z_1$ and $VZ_1$ are lines of $\mathcal{L}$ and since $X_1,V,Z_1\in r^\perp$,
we have that $Z_1^*$  is contained in $r^\perp$; a contradiction. 
(Otherwise, $Z_1^*$ would be a plane through $Z_2$).
Hence $r\in \mathcal{S}$ and
$\mathcal{S}$ is closed under taking reguli.
By  \cite[\S 3.1.2]{BTVM98} and \cite{Luyckx:2001fk}, such a spread of $\Q^-(5,q)$ is necessarily a Hermitian spread of $\Q^-(5,q)$.
\end{proof}

\begin{proof}[Proof of Theorem \ref{thm:Q6embedding}]\label{sect:proofQ6embedding}
First we will translate the hypothesis to the 3-dimensional Hermitian variety $\mathsf{H}(3,q^2)$
via the Barlotti-Cofman-Segre correspondence. So let us fix a non-degenerate hyperplane section
$\Q^-(5,q)$ and consider the set $\mathcal{S}$ of lines of $\mathcal{L}$ that are contained in
$\Q^-(5,q)$. By Lemma \ref{lemma:hermitianspread}, $\mathcal{S}$ is a Hermitian spread of $\Q^-(5,q)$
and so we have the ingredients for the Barlotti-Cofman-Segre correspondence, whereby the spread
$\mathcal{S}$ corresponds to a fixed Hermitian curve $\mathcal{O}$ of $\mathsf{H}(3,q^2)$.  Recall
that the elements of $\mathcal{L}$ not contained in $\Q^-(5,q)$ are mapped to a subset $\Omega$ of
the Baer subgenerators having a point in $\mathcal{O}$.  Also, the affine points of $\Q(6,q)$ are
mapped to the affine points of $\mathsf{H}(3,q^2)\backslash\mathcal{O}$.  We will show that $\Omega$
satisfies the hypotheses of Theorem \ref{thm:constU3}; that is, a generator spanned by $q+1$
elements of $\mathcal{L}$ corresponds to a Baer subplane of $\mathsf{H}(3,q^2)$.  Now by Theorem
\ref{thm:HqQuadric}, we have either (a) $\mathcal{L}$ is the union of lines of the planes of a
spread $\mathcal{S}$ of $\Q(6,q)$, or (b) $\mathcal{L}$ forms the lines of a generalised
hexagon. Case (a) cannot occur as the concurrency graph of $\mathcal{L}$ is connected. So
$\mathcal{L}$ is the lines of a generalised hexagon embedded into $\Q(6,q)$.  Let $P$ be an affine
point of $\Q(6,q)$ and let $P^*$ be the $q+1$ elements of $\mathcal{L}$ incident with $P$. By our
hypothesis, $P^*$ spans a plane $\pi_P$. If this plane were to be incident with an element of
$\mathcal{S}$, then $\pi_P$ would contain more than $q+1$ elements of $\mathcal{L}$ thus implying
that $\pi_P$ would have all of its lines in $\mathcal{L}$; this would then imply that the
concurrency graph of $\mathcal{L}$ is disconnected (see the proof of Lemma
\ref{thm:HqQuadric}). Therefore, $\pi_P$ is not incident with any element of $\mathcal{S}$, and
hence, $\pi_P$ meets $\Q^-(5,q)$ in a transversal line to $q+1$ elements of $\mathcal{S}$.  By the
Barlotti-Cofman-Segre correspondence, $\pi_P$ corresponds to a Baer subplane of $\mathsf{H}(3,q^2)$,
as required.

By Theorem \ref{thm:omega}, $\Omega$ is an orbit of $\mathsf{SU}_3$. Moreover,
this group $\mathsf{SU}_3$ lies within the stabiliser in $\mathsf{PGU}_4(q)$  of a non-degenerate hyperplane,
and so corresponds to a subgroup $\overline{\mathsf{SU}_3}$ of the stabiliser of $\mathcal{S}$.
Now there are $q+1$ split Cayley hexagons whose lines not lying in $\Q^-(5,q)$ form
an orbit under $\overline{\mathsf{SU}_3}$, so it remains to
observe that $\overline{\mathsf{SU}_3}$ has only $q+1$ orbits of size $q(q+1)(q^3+1)$. Indeed, the orbits of
$\overline{\mathsf{SU}_3}$ on lines of $\Q(6,q)$ can be described completely geometrically from
the corresponding orbits of objects in $\mathsf{H}(3,q^2)$ (see Table \ref{BCSrepresentation}). Therefore,
$\Omega$ is the set of lines of some split Cayley hexagon (having a set of lines containing $\mathcal{S}$).

\begin{center}
\begin{table}[H]\footnotesize
\begin{tabular}{p{6cm}|p{6cm}|p{4cm}}
Orbits in $\mathsf{H}(3,q^2)$&Orbits on lines of $\Q(6,q)$&Size\\ \hline
\rowcolor[gray]{0.95}Hermitian curve $\mathcal{O}$ of $\pi_\infty$ & Hermitian spread $\mathcal{S}$ of $\Q^-(5,q)$&$q^3+1$\\
  Affine points $\mathsf{H}(3,q^2)\backslash \pi_\infty$ & Lines of $\Q^-(5,q)$ not in $\mathcal{S}$&$q^2(q^3+1)$\\
\rowcolor[gray]{0.95} Baer subgenerators with no point in $\mathcal{O}$& Affine lines not meeting an element of $\mathcal{S}$ in a totally singular plane &
$q^2(q^2-1)(q^3+1)$\\
Baer subgenerators with a point in $\mathcal{O}$& Affine lines meeting an element of $\mathcal{S}$ in a totally singular plane & $(q+1)\times q(q+1)(q^3+1)$\\ \hline
\end{tabular}
\medskip
\caption{Orbits of $\overline{\mathsf{SU}_3}$ on lines of $\Q(6,q)$.}\label{BCSrepresentation2}
\end{table}
\end{center}

\end{proof}

\section{A connection with Phan theory}\label{sect:phan}

In the theory of linear algebraic groups, if a simply connected simple algebraic group $G$ of type $B_n$,
$C_n$, $D_{2n}$, $E_7$, $E_8$, $F_4$ or $G_2$ has a Curtis-Tits system for its extended Dynkin
diagram then there is a \textit{twisted} version known as a \textit{Phan system} for associated finite groups
corresponding to fixed points of so-called \textit{Frobenius maps} of $G$,
 where the $\mathsf{SL}_2$-subgroups of the Curtis-Tits system are replaced with certain
$\mathsf{SU}_2$-subgroups. This phenomenon has been known since the 1970's to both group theorists
and those working in the theory of twin buildings. In a Curtis-Tits system for a
finite group $G$ (defined over $\GF(q)$), if $K_\alpha$ and $K_\beta$ are two
$\mathsf{SL}_2$-subgroups for two fundamental roots $\alpha$ and $\beta$ joined by a single bond,
then $\langle K_\alpha,K_\beta\rangle$ is isomorphic to $\mathsf{(P)SL}_3(q)$. Whereas in the
corresponding Phan system, a single bond represents an amalgam $\langle K_\alpha,K_\beta\rangle$
isomorphic to $\mathsf{(P)SU}_3(q)$. (See  \cite{BGHS}, \cite{BY} and \cite{Gramlich09} for more on Phan systems).
The geometric model of the split Cayley hexagon that we presented
in this paper was inspired by a \textit{unitary} analogue of the $\mathsf{SL}_3$-model introduced by
Cameron and Kantor \cite{CameronKantor79}.

\begin{table}[H]
\begin{tabular}{c|c}
Curtis-Tits system & Phan system\\
\begin{tikzpicture}[scale=1.4]
\node (a) at (0,0) [circle, draw,inner sep=2pt] {};
\node (b) at (1,0) [circle, draw, inner sep=2pt] {};
\node (c) at (2,0) [circle, draw, inner sep=2pt] {};
\draw (a) to (b);
\draw (b) to (c);
\draw (a.30) to (b.150);
\draw (a.-30) to (b.-150);
\draw (.5,.1) to (.4,0);
\draw (.4,0) to (.5,-.1);
\node [below] at (a.south) {\footnotesize $\alpha$};
\node [below] at (b.south) {\footnotesize $\beta$};
\node [below] at (c.south) {\footnotesize ${-3\alpha-2\beta}$};
\node at (0.5,0.3) {\scriptsize $\mathsf{G}_2(q)$};
\node at (1.5,0.3) {\scriptsize $\mathsf{SL}_3(q)$};
\end{tikzpicture}
&
\begin{tikzpicture}[scale=1.4]
\node (a) at (0,0) [circle, draw,inner sep=2pt] {};
\node (b) at (1,0) [circle, draw, inner sep=2pt] {};
\node (c) at (2,0) [circle, draw, inner sep=2pt] {};
\draw (a) to (b);
\draw (b) to (c);
\draw (a.30) to (b.150);
\draw (a.-30) to (b.-150);
\draw (.5,.1) to (.4,0);
\draw (.4,0) to (.5,-.1);
\node [below] at (a.south) {\footnotesize $\alpha$};
\node [below] at (b.south) {\footnotesize $\beta$};
\node [below] at (c.south) {\footnotesize ${-3\alpha-2\beta}$};
\node at (0.5,0.3) {\scriptsize $\mathsf{G}_2(q)$};
\node at (1.5,0.3) {\scriptsize $\mathsf{SU}_3(q)$};
\end{tikzpicture}\\
\end{tabular}
\medskip
\caption{A summary of the Curtis-Tits and Phan systems for the extended Dynkin diagram of $\mathsf{G}_2(q)$.}
\end{table}

\section{Acknowledgements}

The authors thank Prof Frank De Clerck for various discussions concerning this work (and for
\textit{ontbijten}!), and they also thank Prof Guglielmo Lunardon for his comments on Theorem
\ref{thm:HqQuadric}. We thank Dr {\c{S}}{\"u}kr{\"u} Yal{\c{c}}{\i}nkaya for his expert advice on Phan
systems.  This work was supported by the GOA-grant ``Incidence Geometry'' at Ghent University. The
first author acknowledges the support of a Marie Curie Incoming International Fellowship within the
6th European Community Framework Programme (MIIF1-CT-2006-040360), and the second author
acknowledges the support of a GNSAGA-grant.

We are especially grateful to the anonymous referees whose remarks have greatly improved the
exposition and clarity of this paper.

\end{document}